\documentclass{amsart}

\usepackage{verbatim}
\usepackage{graphicx}
\usepackage{amssymb}
\usepackage{amsfonts}
\usepackage{amsmath}
\usepackage{amsthm}
\usepackage{hyperref}
\usepackage{diagrams}

\def\RR{\mathbb{R}}

\def\NN{\mathbb{N}}

\def\QQ{\mathbb{Q}}
\newcommand{\m}[1]{\Vert #1 \Vert_\infty}

\DeclareMathOperator{\im}{im}
\DeclareMathOperator{\tr}{tr}
\DeclareMathOperator{\shape}{shape}

\DeclareMathOperator{\rank}{rank}
\DeclareMathOperator{\lin}{span}

\newtheorem{theorem}{Theorem}[section]

\newtheorem{lemma}[theorem]{Lemma}
\newtheorem{prop}[theorem]{Proposition}
\newtheorem{cor}[theorem]{Corollary}

\theoremstyle{definition}
\newtheorem{remark}[theorem]{Remark}
\newtheorem{ex}[theorem]{Example}

\usepackage{mathtools}

\begin{document}
\title{Archimedean classes of matrices over ordered fields}

\author{J. Cimpri\v c}

\address{University of Ljubljana, Faculty of Mathematics and Physics, Department of Mathematics, Jadranska 21, SI-1000 Ljubljana, Slovenia}

\email{cimpric@fmf.uni-lj.si}

\thanks{Supported by grant P1-0222 of the Slovenian Research Agency}



\date{\today}

\begin{abstract}
Let $(F,\le)$ be an ordered field and let $A,B$ be square matrices over $F$ of the same size.
We say that $A$ and $B$ belong to the same archimedean class if there exists an integer $r$
such that the matrices $r A^T A-B^T B$ and $r B^T B-A^T A$ are positive semidefinite with respect to $\le$.
We show that this is true if and only if $A=CB$ for some invertible matrix $C$
such that all entries of $C$ and $C^{-1}$ are bounded by some integer. 
We also show that every archimedean class contains a row echelon form and that its shape 
and archimedean classes (in $F$) of its pivots  are uniquely determined.
For matrices over fields of formal Laurent series we construct a canonical representative
in each archimedean class. The set of all archimedean classes is shown to have a natural 
lattice structure while the semigroup structure does not come from matrix multiplication.
Our motivation comes from noncommutative real algebraic geometry and 
noncommutative valuation theory. 
\end{abstract}

\maketitle

\section{Introduction}

The notion of the natural valuation of an ordered field was introduced by Baer in \cite{baer}.
Through \cite{as} it motivated Krull to introduce valuations with non-archimedean value groups in \cite{krull}.
Krull's valuation theory was extended to skew-fields by Schilling \cite{schilling} and natural valuations
of ordered skew-fields were studied by Conrad \cite{conrad} and Holland \cite{holland}.
For matrices over skew-fields the valuation theory was developed in \cite{dubrovin}, \cite{morandi} and
\cite{tignol} but it can be shown that orderings exist only in the $1 \times 1$ case.
On the other hand, partial orderings also exist in other cases but their natural valuations
have only been studied in the commutative $1 \times 1$ case; see \cite{becker}. 
It would be interesting to study  other cases, too.

We will concentrate on the simplest case, namely matrices over ordered fields with transpose as involution
and positive semidefinitness as partial ordering, because the theory is already nontrivial
and the results may be of interest in linear algebra. We define a relation $\succeq_n$ on $M_n(F)$ 
by $A \succeq_n B$ if and only if there exist an integer $r$ such that $r B^T B - A^T A$ is positive semidefinite 
with respect to the ordering of $F$. This relation is reflexive and transitive, but it need not be antisymmetric. 
The corresponding equivalence relation $\sim_n$ on $M_n(F)$ defined by $A \sim_n B$ if and only if $A \succeq_n B$ and $B \succeq_n A$
is called \textit{archimedean equivalence} and the elements of the factor set $M_n(F)/\!\sim_n$
are called \textit{archimedean classes}. The canonical projection $v_n \colon M_n(F) \to M_n(F)/\!\sim_n$ 
is called the \textit{natural valuation}.

The paper is organized as follows. In section \ref{classic} we recall the construction of the natural valuation
of an ordered field and generalize it to matrices over ordered fields. Our construction does not produce a value
function in the sense of Morandi \cite{morandi}, so we carefully explain the relationship between the two.
In section \ref{secmain} we show that each archimedean class in $M_n(F)$ contains a row echelon form and that 
its shape and the natural valuations of its pivots are uniquely determined.
In section \ref{seclattice} we show that the relation $\succeq_n$ induces a lattice structure on $M_n(F)/\!\sim_n$.
In section \ref{secbounded} we characterize relations $\succeq_n$ and $\sim_n$ by divisibility.
In section \ref{seccan} we try to find a canonical representative in each archimedean class.
This works for matrices over formal Laurent series fields but not in general. In section \ref{seclast} 
we discuss different ways of introducing a semigroup structure on $M_n(F)/\!\sim_n$.

\section{Preliminaries on natural valuations}
\label{classic}

We will recall the construction of the natural valuations of an ordered field and generalize it to an ordered ring with involution.
We will illustrate our construction on the matrix ring $M_n(F)$ with transpose involution ordered in two different ways. The natural valuation of the first 
ordering is a value function in the sense of Morandi while the  natural valuation of the second ordering does not have an analogue
in noncommutative valuation theory.

\subsection{Natural valuations of ordered fields}
\label{subsecfields}

Let $(F,\le)$ be an ordered field. Clearly, $\operatorname{char} F=0$, and so $F$ contains $\QQ$. 
For every $a,b \in F$ write $a \succeq b$ iff there exists $r \in \NN$ such that 
$\vert a \vert \le r \vert b \vert$ (or equivalently, $a^2 \le r^2 b^2$). Write $a \sim b$ iff $a \succeq b$ and $b \succeq a$ and note that this relation is
a congruence on the multiplicative semigroup of $F$. The congruence classes are also called \textit{archimedean classes}
of $(F,\le)$. The factor semigroup $F/\!\sim$ is linearly ordered
by $[a] \succeq [b]$ iff $a \succeq b$. The canonical projection from $F$ to $F/\!\sim$ will be denoted
by $v$ and called the \textit{natural valuation} of $\le$. We will write $F/\!\sim\, =\Gamma \cup \{\infty\}$ where
$\infty:=\{0\}$ is the congruence class of zero and $\Gamma$ is the set of all other congruence classes. Clearly,
$\Gamma$ is an abelian group with $[a]+[b]:=[ab]$ as operation and $0:=[1]$ as neutral element. We will call it the \textit{value group} of $v$.

Let us now briefly sketch an alternative construction. We say that an element $a \in F$ is \textit{bounded} if there exists $r \in \NN$
such that $-r \le a \le r$. An element $a \in F$ is \textit{infinitesimal} if $-r \le a \le r$ for every $r \in \NN$.
The set $V$ of all bounded elements is a valuation subring of $F$ and the set $m$ of all infinitesimal elements is the only
maximal ideal of $V$. The set $U:=V \setminus m$ is a subgroup of the multiplicative group $F^\times=F \setminus \{0\}$
and the factor group $\Gamma:=F^\times /U$ is linearly ordered by $a U \succeq b U$ iff  $\frac{a}{b} \in V$.
The natural valuation of $(F,\le)$ is then the canonical projection $v \colon F^\times \to \Gamma$ 
extended to $F$ by $v(0)=\infty$ where $\infty \not\in \Gamma$ is larger from all elements from $\Gamma$.

For every ordered group $G$, the field $\RR(G))$ of formal Laurent series can be ordered by the sign of the lowest
nonzero coefficient. The corresponding value group is $G$ and the natural valuation assigns to each element the 
least element of $G$ with nonzero coefficient. Hahn's embedding theorem for ordered fields says that 
every ordered field with value group $\Gamma$ has an order-preserving embedding into $\RR((\Gamma))$.
See \cite{hahn} for the origins; a complete proof appeared much later. In \cite{kuhlmann} it is shown
that the real closure of an ordered field has a truncation closed embedding into $\RR((\bar{\Gamma}))$
(where $\bar{\Gamma}$ is the division hull of $\Gamma$) which maps the field into $\RR((\Gamma))$.

An ordered field is \textit{archimedean} iff it has no unbounded elements iff it has no nonzero infinitesimal elements
iff $\Gamma$ has only one element iff it is a subfield of $\RR$. 

\subsection{Ordered rings with involution}
The construction from the previous section can be extended to ordered rings with involution.
Let $R$ be a ring with involution $\ast$ and let $\le$ be a relation on the subset
$S := \{a \in R \mid a^\ast=a\}$ which is reflexive, transitive and satisfies $0 \le a^\ast a$ for every $a \in R$.
We define two relations  on $R$ by
\begin{center}
$a \succeq b$ iff there is  $r \in \NN$ such that $a^\ast a \le r b^\ast b$,

\smallskip

$a \sim b$ iff $a \succeq b$ and $b \succeq a$.
\end{center}
Clearly, $\succeq$ is reflexive and transitive and $\sim$ is an equivalence relation.
As above, we say that the canonical projection $v \colon R \to R /\!\sim$ is
\textit{the natural valuation} of $\le$.
Consider the equivalence class $\infty :=v(0)$.
Its complement is denoted by $\Gamma$ and called \textit{the value set} of $v$.
Whenever we work with several rings or several orderings, we will add a suitable index to $\le$, $\succeq$, $\sim$, $v$ or $\Gamma$.

Note that $v(a)=\infty$ iff $a^\ast a = 0$, so $\Gamma$ can be empty.
Lemma \ref{bad} will show that the factor set $R /\!\sim \;=\Gamma \cup \{\infty\}$ need not have any algebraic structure
because $\sim$ need not be a congruence relation on the multiplicative semigroup of $R$
and it need not be compatible with the involution.
Moreover, the partial ordering of $R /\!\sim$ induced by $\succeq$ need not be a linear ordering.

\begin{ex}
If $R$ is a partially ordered commutative ring with positive squares then $\Gamma \cup \{\infty\}$ has the structure 
of a partially ordered commutative semigroup. In particular, if $R$ is a partially ordered field with positive squares
then $\Gamma$ has the structure of a partially ordered abelian group. See \cite{becker} for additional information.
\end{ex}

Let $R$ be the ring $M_n(F)$ of all $n \times n$ matrices with transpose as involution.
Then $S$ is the vector space $S_n(F)$ of all symmetric $n \times n$ matrices. 
We will consider two orderings of $S$: $C \le D$ iff $\tr C \le \tr D$ in subsection \ref{ncvalsec}
and $C \le D$ iff $D-C$ is positive semidefinite in subsection \ref{maincasesec}. The first ordering is uncommon
but its natural valuation is the most common valuation on $R$. The second ordering is the most common
ordering of $R$ but its natural valuation is very uncommon. The rest of the paper will 
then give more details about this uncommon  valuation.

\subsection{Noncommutative valuation theory}
\label{ncvalsec}

Let $(F,\le_F)$ be an ordered field with natural valuation $v_F \colon F \to \Gamma_F \cup \{\infty\}$
and let $n$ be a natural number. We assume that $S_n(F)$ is ordered by $C \le D$ iff $\tr C \le_F \tr D$.
The corresponding ordering $\succeq$ of $M_n(F)$ is then $A \succeq B$ iff $\tr A^T A \le_F r \tr B^T B$
for some $r \in \NN$.  We will denote the natural valuation of $\le$ by $w$.

The mapping $c \to c I_n$ identifies $F$ with a subset of $M_n(F)$. Moreover, we have that
$c I_n \succeq d I_n$ iff $c \succeq_F d$. It follows that $\Gamma_F$ can be identified
with a subset of $\Gamma$. To show that $\Gamma=\Gamma_F$, it suffices to observe that
for every matrix $A=[a_{ij}] \in M_n(F)$,
\begin{equation}
\label{ncvaleq1}
A \sim \Vert A \Vert_\infty \cdot I_n,
\end{equation}
where $\Vert A \Vert_\infty:=\max_{i,j} \vert a_{ij} \vert$
is the max norm of $A$ and $I_n$ is the identity matrix.
Furthermore, relation \eqref{ncvaleq1} implies that for every $A \in M_n(F)$,
\begin{equation}
\label{ncvaleq2}
w(A)=v_F(\Vert A \Vert_\infty).
\end{equation}

Since $v_F$ is the natural valuation, $v_F(\max\{x,y\})=v_F(x+y)=\min\{v_F(x),v_F(y)\}$ for every $x,y \in F^{\ge 0}$.
It follows that for every matrix $A=[a_{ij}]$ we have that
\begin{equation}
\label{ncvaleq3}
v_F(\Vert A \Vert_\infty)=\min_{i,j} v_F(a_{ij}).
\end{equation}

Relations \eqref{ncvaleq2} and \eqref{ncvaleq3} imply that the natural valuation $w \colon M_n(F) \to \Gamma \cup \{\infty\}$
satisfies \cite[Definition 1.2]{tignol}, i.e. $w$ is a $v_F$-value function on $M_n(F)$ 
which for every $A,B \in M_n(F)$ satisfies $w(A)+w(B) \le w(AB)$ and $w(A^T A)=2 w(A)$.
By the proof of \cite[Theorem 2.2]{tignol}, $w$ is the only such function.
Therefore $w$ can be considered as the canonical extension of $v_F$ from $F$ to $M_n(F)$.

\subsection{The main example}
\label{maincasesec}

A matrix $A \in S_n(F)$ is \textit{positive semidefinite} iff $v^T A v \ge 0$ for every $v \in F^n$.
The set of all positive semidefinite matrices will be denoted by $S_n^+(F)$. 
It defines a partial ordering $\le_n$ of $S_n(F)$ by $A \le_n B$ iff $B-A \in S_n^+(F)$.
The corresponding relations on $M_n(F)$ are defined by $A \succeq_n B$ iff $A^T A \le_n r B^T B$ for some $r \in \NN$
and $A\sim_n B$ iff $A \succeq_n B$ and $B \succeq_n A$. We will call $\sim_n$ \textit{archimedean equivalence}.
The natural valuation of $\le_n$ will be denoted by $v_n \colon M_n(F) \to M_n(F)/\!\sim_n$ where
we decompose $M_n(F)/\!\sim_n$ into $\Gamma_n$  and $\infty =\{0_n\}$. 

Note that $A^T A \le r B^T B$ implies $\tr A^T A \le _F r \tr B^T B$, thus $v_n(A) \succeq_n v_n(B)$ implies $w(A) \succeq w(B)$.
Therefore, we have a mapping $\phi \colon \Gamma_n \cup \{\infty\} \to \Gamma_F \cup \{\infty\}$ which is order-preserving and 
makes the following diagram commutative:

\begin{diagram}
M_n(F) & \rTo{v_n}{} & \Gamma_n \\
\dTo{\Vert \cdot \Vert_\infty}{} &  \rdTo{w}{} & \dTo{}{\phi} \\
F & \rTo{v_F}{} & \Gamma_F
\end{diagram}

As above, the mapping $v_F(c) \mapsto v_n(c I_n)$ identifies $\Gamma_F$ with a subset of $\Gamma_n$.
In other words, we can consider $v_n$ as a refinement of $w$. Let us show that the properties of $v_n$
are much worse than the properties of $w$.

\begin{lemma}
\label{bad}
Notation from above. If $n \ge 2$, then
\begin{enumerate}
\item $\sim_n$ is not a congruence relation on the multiplicative semigroup of $M_n(F)$.
(Namely, $A \sim_n B$ always implies $AC \sim_n BC$ but not $CA \sim_n CB$.)
\item $\succeq_n$ is not a linear ordering of $\Gamma_n \cup \{\infty\}$.
\item There exists $A \in M_n(F)$ such that $v_n(A) \ne v_n(A^T)$.
\end{enumerate}
\end{lemma}

\begin{proof}
To prove (1) for $n=2$, note that
$$\left[ \begin{array}{cc} 0 & 1 \\ 1 & 0 \end{array} \right] \sim_2 \left[ \begin{array}{cc} 1 & 0 \\ 0 & 1 \end{array} \right]$$
but 
$$\left[ \begin{array}{cc} 1 & 0 \\ 0 & 0 \end{array} \right]\left[ \begin{array}{cc} 0 & 1 \\ 1 & 0 \end{array} \right] 
\not\sim_2 \left[ \begin{array}{cc} 1 & 0 \\ 0 & 0 \end{array} \right]\left[ \begin{array}{cc} 1 & 0 \\ 0 & 1 \end{array} \right] .$$
To prove (2) and (3) for $n=2$, note that $A=\left[ \begin{array}{cc} 0 & 1 \\ 0 & 0 \end{array} \right]$ satisfies
$A \not\succeq_2 A^T$ and $A^T \not\succeq_2 A$. To get examples for larger $n$ just add some zero rows and columns.
\end{proof}

Lemma \ref{ker} characterizes relations $\succeq_n$ and $\sim_n$ for archimedean fields. Note that
\begin{equation}
\label{trace}
A \le_n (\tr A)  I_n
\end{equation}
for every $A \in S_n^+(F)$ since $(\tr A) I_n-A= \sum_{1 \le i<j \le n} (E_{ij}-E_{ji})^T A (E_{ij}-E_{ji})$.

\begin{lemma}
\label{ker}
Notation from above. Pick any $A,B \in M_n(F)$. If $A \succeq_n B$ then $\ker A \supseteq \ker B$.
If $F$ is an archimedean field then the converse is also true.
If $F$ is a non-archimedean field then the converse fails already for $n=1$.
\end{lemma}

\begin{proof} 
The first claim is clear. To prove the second claim note that $\ker A \supseteq \ker B$
implies that there is some $C$ of appropriate size such that $A=CB$. The trace inequality
\eqref{trace} implies that $A^T A=B^T C^T C B \le_n \tr(C^T C) B^T B$. If $F$
is an archimedean field, then $\tr C^T C$ is bounded by some natural number, so $A \succeq_n B$.
If $F$ is a non-archimedean field, then it contains some unbounded element $t$. Note that
$\ker [t] =\ker [1]=0$ but $[t] \not\succeq_1 [1]$.
\end{proof}

\section{Row echelon forms}
\label{secmain}

We will show that each archimedean class contains a row echelon form whose shape and natural valuations of its pivots are
uniquely determined.

\subsection{$QR$ decomposition} 
Let $M_{\ast,n}(F):=\bigcup_{m=1}^\infty M_{m,n}(F)$ be the set of all matrices over $F$ with $n$ columns and arbitrary many rows.
For technical reasons we will also study extensions of the relations $\succeq_n$ and $\sim_n$ from $M_n(F)$ to $M_{\ast,n}(F)$.

For every $A, B \in M_{\ast,n}(F)$ we write, as above, 
\begin{center}
$A \succeq_n B$ iff $A^T A \le_n r B^T B$ for some $r \in \NN$ and 

\smallskip

$A \sim_n B$ iff $A \succeq_n B$ and $B \succeq_n A$. 
\end{center}

By the Gram-Schmidt algorithm, every subspace of $F^n$ has an orthogonal basis with respect to the standard inner product. 
However, we do not always have an orthonormal basis unless $F$ is pythagorean. Instead, we will use normalization 
in the $\ell^\infty$-norm $\Vert v \Vert_\infty := \max_i \vert v_i \vert$. 

We say that a matrix $C$ is a \textit{row echelon form} if all zero rows of $C$ are at the bottom of $C$ and if
for each $i$ the first nonzero element in the $(i+1)$-th nonzero row is on the right-hand side of the first nonzero element 
in the $i$-th nonzero row. The first nonzero element in the $i$-th nonzero row is also called \textit{the $i$-th pivot}.

Lemma \ref{qr} is a variant of $QR$ decomposition.

\begin{lemma}
\label{qr}
For every $A \in M_{m,n}(F)$ with rank $r \ge 1$ there exists matrices $Q \in M_{m,r}(F)$ and $R \in M_{r,n}(F)$ such that
\begin{enumerate}
\item $A=QR$.
\item The columns of $Q$ are orthogonal and $\ell^\infty$-normalized.
\item $R$ is a row echelon form with positive pivots and no zero rows.
\end{enumerate}
Moreover, $Q$ and $R$ are uniquely determined by $A$. We also have that $Q \sim_r I_r$.
\end{lemma}

\begin{proof}
Let $v_i$ be the $i$-th column of $A$ for each $i=1,\ldots,n$ and let $r$ be the rank of $A$.
Let $k_1$ be the first index such that $v_{k_1} \ne 0$, $k_2$ the first index such that $v_{k_2} \not\in \lin\{v_{k_1}\}$,
$k_3$ the first index such that $v_{k_3} \not\in\lin\{v_{k_1},v_{k_2}\}$, etc. Now  set $w_1'=v_{k_1}$ and 
$w_i' = v_{k_i}-\sum_{j=1}^{i-1} \frac{\langle v_{k_i},w_i' \rangle}{\langle w_i',w_i' \rangle} w_i'$
for $i=2,\ldots,r$. Write $w_i=w_i'/\Vert w_i' \Vert_\infty$ for $i=1,\ldots,r$ and note that $Q:=[w_1 \ldots w_r]$
satisfies (2). Now pick $c_{ij} \in F$ such that $v_j=\sum_{i=1}^r c_{ij} w_i$ for $j=1,\ldots,n$
and note that $R:=[c_{ij}]$ satisfies (1) and (3). 
The matrix $D:=Q^T Q$ is diagonal since $w_i$ are orthogonal and it satisfies $I_r \le_n D \le_n m I_n$
since $w_i$ are $\ell^\infty$-normalized. It follows that $Q \sim_n I_r$. Suppose that $Q R =  Q' R'$, 
where $Q=[w_1 \ldots w_r]$ and $Q'=[z_1 \ldots z_r]$ satisfy (1) and $R=[c_{ij}]$ and $S=[d_{ij}]$ satisfy (2).
Let $c_{i,k_i}$ and $d_{j,l_j}$ be the pivots of $R$ and $R$'.
The $t$-th column of $QR=Q'R'$ is $c_{1t} w_1+\ldots+c_{tt} w_t=d_{1t} z_1+\ldots+d_{tt} z_t$.
For $t=\min\{k_1,l_1\}$ one gets that $w_1=z_1$; for $t=\min\{k_2,l_2\}$ one gets that $w_2=z_2$; and so on.
\end{proof}

The natural embedding $j \colon M_n(F) \to M_{\ast,n}(F)$ satisfies the property $A \succeq_n B$ iff $j(A) \succeq_n j(B)$ for every $A,B \in M_n(F)$.
In other words, $j$ is an $o$-\textit{embedding}. It follows that $j$ induces a mapping $j' \colon M_n(F)/\!\sim_n \to M_{\ast,n}(F)/\!\sim_n$ which is also
an $o$-embedding. Corollary \ref{square} will show that every element of $M_{\ast,n}(F)/\!\sim_n$ contains a square matrix
which implies that the mapping $j'$ is onto. Therefore, $j'$ is an isomorphism of partially ordered sets.

\begin{cor}
\label{square}
Every archimedean class of $M_{\ast,n}(F)$ contains a square matrix which is a row echelon form.
\end{cor}

\begin{proof}
Pick any $A \in M_{\ast,n}(F)$ of rank $r$. The claim is clear if $r=0$. Otherwise
we can decompose $A=QR$ where $Q \sim_r I_r$ and $R$ is a row echelon form.
It follows that $QR \sim_n R$, so $R$ belongs to the archimedean class of $A$.
Since $R \sim_n \left[ \begin{array}{c} R \\ 0 \end{array} \right]$ for any number of
zero rows, the claim follows.
\end{proof}

\subsection{Shape}

For every row echelon form $C \in M_{k,n}(F)$ we define its \textit{shape}
$$\shape(C)=\{(i,j) \in \NN_k \times \NN_n \mid \exists j_0 \le j \colon c_{ij_0}  \ne 0\}.$$
where $\NN_n:=\{1,\ldots,n\}$. 
It consists of all positions that lie above the \lq\lq staircase''.

\begin{prop}
\label{shape}
Let $A,B \in M_{\ast,n}(F)$ be row echelon forms. If $A \succeq B$ then $\shape(A) \subseteq \shape(B)$.
Moreover, for every pivot position $(i,k_i)$ of $\shape(B)$ we have that $v(a_{i,k_i}) \succeq v(b_{i,k_i})$.
\end{prop}

\begin{proof}
Suppose that for some row echelon forms $A,B \in M_{\ast,n}(F)$ we have that $v_n(A) \succeq v_n(B)$ but $\shape(A) \not\subseteq \shape(B)$.
Let $i$ be the smallest integer such that the $i$-th row of $\shape(A)$ is not contained in $\shape(B)$ (which can have less than $i$ rows)
and let $j$ be the smallest integer such that $(i,j) \in \shape(A)$. 
It follows that $(i,j)$ is a pivot position of $\shape(A)$ (so $a_{ij} \ne 0$) and $(i,j) \not\in \shape(B)$
and the $j$-th column of $\shape(B)$ does not contain any pivot position of $\shape(B)$ 
(otherwise $\shape(A)$ would have a step of size $\ge 2$.)

Recall that for every elementary matrix $E$, $C \to CE$ is an elementary column transformation of the matrix $C$.
Let us perform the standard Gauss algorithm on the columns of $B$. We use the first pivot to kill all other
elements in the first row of $B$, then we use the second pivot to kill all other elements in the second row of $B$ and so on.
Let $Q$ be the product of all elementary matrices that we used in this procedure. Note that $\tilde{B}=BQ$ has the same shape
and the same pivots as $B$ but all non-pivot elements are zero. In particular, the $j$-th column of $\tilde{B}$ is zero.
We claim that the $j$-th column of $\tilde{A}:=AQ$ is also zero. 
The assumption $v_n(A) \succeq_n v_n(B)$ implies that there exists $r \in \NN$ such that
$A^T A \le r B^T B$. It follows that $\tilde{A}^T \tilde{A} \le r \tilde{B}^T \tilde{B}$ which implies the claim.

Finally, note that the $(i,j)$-th element of $\tilde{A}$ is $a_{ij}$ because the $i$-th pivot of $B$ (if it exists) 
lies on the right-hand side of $a_{ij}$, and so the elementary column transformations from the product $Q$
did not act on $a_{ij}$. Since the $j$-th column of $\tilde{A}$ is zero, it follows that $a_{ij}=0$
which is not possible  by the choice of $(i,j)$. This contradiction finishes the proof of the first part.

To prove the second part, let $\tilde{A}$ and $\tilde{B}$ be as above. Write $u$ for the $k_i$-th column of $\tilde{A}$
and $v$ for the $k_i$-th column of $\tilde{B}$. As above, $\tilde{A}^T \tilde{A} \le r \tilde{B}^T \tilde{B}$ implies that
$u^T u \le r v^T v$ for some $r \in \NN$. Since $u=(u_1,\ldots,u_{i-1},a_{i,k_i},0,\ldots,0)^T$ for some $u_1,\ldots, u_{i-1} \in F$
and $v=(0,\ldots,0,b_{i,k_i},0,\ldots,0)$, it follows that $a_{i,k_i}^2 \le r b_{i,k_i}^2$, i.e. $v(a_{i,k_i}) \succeq v(b_{i,k_i})$.
\end{proof}

\begin{cor}
\label{corshape}
Let $A,B \in M_{\ast,n}(F)$ be row echelon forms. If $v_n(A) = v_n(B)$ then $\shape(A) = \shape(B)$
and for each $i$, the natural valuation of the $i$-th pivot of $A$ is equal to the natural
valuation of the $i$-th pivot of $B$.
\end{cor}

Corollary \ref{corshape} implies that we can define the shape of an archimedean class.

\section{$\Gamma_n$ is a lattice}
\label{seclattice}

We will show that the partial ordering induced by $\succeq_n$ on the value set $\Gamma_n$  is in fact a lattice ordering,
i.e. each finite subset of $\Gamma_n$ has supremum and infimum.

\subsection{Lattice structure of positive semidefinite matrices}

Let $(F, \ge)$ be an ordered field. For every $A,B \in S_n^+(F)$ write
\begin{center}
$A \sqsupseteq_n B$ iff $A \le r B$ for some $r \in \NN$
\end{center}
and
\begin{center}
$A \approx_n B$ iff $A \sqsupseteq_n B$ and $B \sqsupseteq_n A$.
\end{center}
The mapping $i \colon M_{\ast,n}(F) \to S_n^+(F), A \mapsto A^T A$ satisfies
$A \succeq_n B$ iff $i(A) \sqsupseteq_n i(B)$, i.e. $i$ is an $o$-embedding.
Therefore $i$ induces a mapping $i' \colon M_{\ast,n}(F)/\!\sim_n \to S_n^+(F)/\!\approx_n$ which is also an $o$-embedding. 
The plan is to show that $S_n^+(F)/\!\approx_n$ is a lattice and then pull this result back to $M_{\ast,n}(F)/\!\sim_n$
which is isomorphic to $\Gamma_n \cup \{\infty\}$. One can also deduce Lemma \ref{ker} from Lemma \ref{kerpsd} this way.

\begin{lemma}
\label{kerpsd}
Suppose that $A,B \in S_n^+(F)$. If $A \sqsupseteq_n B$ then also $\ker A \supseteq \ker B$.
If $F$ is an archimedean field, then we also have the converse. For non-archimedean
fields the converse already fails for $n=1$. 
\end{lemma}

\begin{proof}
If $v \in \ker B$ and $A \sqsupseteq_n B$ then clearly $v^T A v=0$. By \cite[Corollary 2.4]{lam},
there exists an invertible $P \in M_n(F)$ and a diagonal $D \in S_n(F)$ with entries $d_i$ such that $A=P^T DP$.
It follows that $\sum_i d_i (Pv)_i^2=0$. Since $0 \le_n A$, also $0 \le_n D$, which implies that 
$d_i (Pv)_i=0$ for each $i$. Therefore $DPv=0$, and so $v \in \ker A$.

Suppose now that $F$ is archimedean and $\ker A \supseteq \ker B$. Pick an invertible $Q \in M_n(F)$ such that 
$E:=Q^T B Q$ is a diagonal matrix with nonzero entries $e_1,\ldots,e_k$.
Write $C:=Q^T A Q$ and note that $\ker C \supseteq \ker E$ implies that $c_{ij}=0$ if either $i>k$ or $j>k$.
It follows that $C =M \oplus 0_{n-k} \le_n (\tr M) I_k \oplus 0_{n-k} \le u D$ where $u \in \NN$
is such that $\frac{\tr M}{\min\{e_1,\ldots,e_k\}} \le u$. Therefore $A \le_n uB$.
\end{proof}

By the Gram-Schmidt algorithm (or \cite[Proposition 1.3]{lam}), every subspace $U$ of $F^n$ satisfies $U \oplus U^{\perp}=F^n$.
It follows that for every  $C \in S_n(F)$ we have that $F^n=\im C \oplus \ker C$.
We define its \textit{Moore-Penrose inverse} $C^+=(C|_{\im C})^{-1} \oplus 0|_{\ker C}$.
For every $A,B \in S_n^+(F)$ we define their \textit{parallel sum} $A:B=A(A+B)^+B$. 
The basic properties of $A:B$ from \cite{ad} carry over from $\RR$ to general ordered fields.
In particular, we have the following generalization of \cite[Corollary 21]{ad}.

\begin{lemma}
\label{parlem}
If $A,B,C \in S_n^+(F)$ and $A \le_n B$ then $A:C \le_n B:C$.
\end{lemma}

\begin{proof}
Note that $\ker(A+C) \subseteq \ker C$. Namely, $(A+C)x=0$ implies that
$x^T A x=x^T C x=0$ which implies that $Cx=0$ as in the proof of Lemma \ref{kerpsd}.
It follows that $\im C \subseteq \im (A+C)$, which implies that $C(A+C)^+(A+C)=C=(A+C)(A+C)^+C$.
If $A \le_n B$, these identities imply that $0 \le_n C(B+C)^+(B-A)(B+C)^+C+C((B+C)^+-(A+C)^+)(A+C)((B+C)^+-(A+C)^+)C
=C(A+C)^+C-C(B+C)^+C
=C-C(B+C)^+C-(C-C(A+C)^+C)
=B:C - A:C$.
\end{proof}

\begin{prop}
\label{lat1}
For every $[A],[B] \in S_n^+(F)/\!\approx_n$, their greatest lower bound and 
their least upper bound are given by the formulas
$$[A] \sqcap_n [B]=[A+B] \qquad \text{and} \qquad [A] \sqcup_n [B]=[A:B].$$
\end{prop}

\begin{proof}
Pick $A,B,C \in S_n^+$. Since $A \le_n A+B$ and $B \le_n A+B$ we have that $A \sqsupseteq_n A+B$ and $B \sqsupseteq_n A+B$.
If $A \sqsupseteq_n C$ and $B \sqsupseteq_n C$, then $A \le r C$ and $B \le s C$ for some $r,s \in \NN$.
If follows that $A+B \le (r+s)C$, thus $A+B \sqsupseteq_n C$. This proves the first part.
To prove the second part, note that $A:B \le_n A$ and $A:B \le_n B$ by \cite[Lemma 18]{ad}, 
which implies that $A:B \sqsupseteq_n A$ and $A:B \sqsupseteq_n B$.
If $C \sqsupseteq_n A$ and $C \sqsupseteq_n B$ then $C \le tA$ and $C \le tB$ for some $t \in \NN$.
By Lemma \ref{parlem}, we have that $\frac12 C=C:C \le_n (tA):C \le_n (tA):(tB)=t(A:B)$
which implies that $C \sqsupseteq_n A:B$.
\end{proof}

For every $A,B \in S_n^+(F)$, we clearly have that $\ker(A+B)=\ker A \cap \ker B$ and by \cite[Lemma 3]{ad} we also have that $\ker(A:B)=\ker A + \ker B$.
By Lemma \ref{kerpsd}, we can define the kernel of an equivalence class by $\ker [A]:= \ker A$. By Proposition \ref{lat1}, $\ker$ is a lattice homomorphism
from $(S_n^+(F)/\approx_n,\sqcup_n,\sqcap_n)$ to $(\mathrm{Subspaces}(F^n),+,\cap)$.

\subsection{Lattice structure of rectangular matrices}

Proposition \ref{lat2} shows that $\Gamma_n \cup \{\infty\}$ is a lattice.

\begin{prop}
\label{lat2}
For every two elements $v_n(A),v_n(B) \in \Gamma_n \cup \{\infty\}$ their greatest lower bound and their least upper bound are given by the formulas
$$v_n(A) \wedge_n  v_n(B)=v_n(\left[\begin{array}{c} A \\ B \end{array} \right]) 
\text{ and } 
v_n(A) \vee_n  v_n(B)=v_n(\left[\begin{array}{c} A(A^T A+B^T B)^+ B^T B \\ B(A^T A+B^T B)^+ A^T A \end{array} \right]).$$
\end{prop}

\begin{proof}
Let $i'$ be as above. By Lemma \ref{lat1}, we have that
$$i'(v_n(\left[\begin{array}{c} A \\ B \end{array} \right]))=[A^T A+B^T B]_n = [A^T A] \sqcap_n [B^T B]=i'(v_n(A)) \sqcap_n i'(v_n(B))$$
and
$$i'(v_n(\left[\begin{array}{c} A(A^T A+B^T B)^+ B^T B \\ B(A^T A+B^T B)^+ A^T A \end{array} \right]))= $$
$$= [A^T A(A^T A+B^T B)^+ B^T B (A^T A+B^T B)^+ A^T A+ \quad$$ 
$$\qquad +B^T B (A^T A+B^T B)^+ A^T A (A^T A+B^T B)^+ B^T B]_n= $$
$$= [(A^T A+B^T B)(A^T A+B^T B)^+ (A^T A : B^T B)]_n=$$ 
$$=[A^T A : B^T B]_n= [A^T A] \sqcup_n [B^T B]_n= i'(v_n(A)) \sqcup_n i'(v_n(B)).$$
Since $i'$ is an $o$-imbedding, this implies the formulas.
\end{proof}

By Propositions \ref{lat1} and \ref{lat2}, the mapping $\ker \circ i'$ is a lattice homomorphism from
$(M_{\ast,n}(F)/\!\sim_n,\vee_n,\wedge_n)$ to $(\mathrm{Subspaces}(F^n),+,\cap)$.

\begin{cor}
For every $A,B \in M_{\ast,n}(F)$ of the same size 
we have that $$v_n(A+B) \succeq v_n(A) \wedge_n v_n(B).$$
\end{cor}

\begin{proof}
This follows from Lemma \ref{lat2} and 
$$0 \le_n (A-B)^T (A-B) =2\left[\begin{array}{c} A \\ B \end{array} \right]^T \left[\begin{array}{c} A \\ B \end{array} \right]-(A+B)^T (A+B) \qedhere$$
\end{proof}

\begin{cor}
For every $A.B \in M_{\ast,n}(F)$, we have $\shape(v_n(A)) \cup \shape(v_n(B)) \subseteq \shape (v_n(A) \wedge_2 v_n(B))$.
The inclusion can be strict.
\end{cor}

\begin{proof} 
The first part follows from Proposition \ref{shape}. To prove the second part, note that the relation
$E_{11}^T E_{11}+E_{12}^T E_{12}=I_2^T I_2$ implies that $v_2(E_{11}) \wedge_2 v_2(E_{12})=v_2(I_2)$.
Therefore, $\shape(v_2(E_{11}) \wedge_2 v_2(E_{12}))=\shape(v_2(I_2))=\{(1,1),(1,2),(2,2)\}$.
On the other hand, $\shape(v_n(A)) \cup \shape(v_n(B))=\{(1,1),(1,2)\}$.
\end{proof}

\section{Bounded and bibounded elements}
\label{secbounded}

Let $(F,\le)$ be an ordered field and $n$ an integer. 
We say that a matrix $A \in M_{\ast,n}(F)$ is \textit{bounded} (w.r.t. $\le_n$) if $A \succeq_n I_n$ where $I_n$ is the identity matrix.
We say that $A \in M_{\ast,n}(F)$ is \textit{bibounded} if $A \sim_n I_n$. We will also use this terminology for scalars and vectors which
can be identified with elements of $M_{\ast,1}(F)$. A scalar $a$ is bounded iff $v(a) \ge 0$. It is bibounded iff $v(a)=0$. 
Every vector $v$ clearly satisfies $v \sim_1 \Vert v \Vert_\infty$ where $\Vert v \Vert_\infty$ is the $\ell^\infty$ norm of $v$.
It follows that a vector is bounded iff all its components are bounded
and that a vector is bibounded iff its $\ell^\infty$ norm is bibounded. 
We can make every nonzero vector bibounded by dividing it with its $\ell^\infty$ norm.

\subsection{A characterization of $\succeq_n$ and $\sim_n$}

The aim if this section is to characterize the relations $\succeq_n$ and $\sim_n$ in terms of divisibility.

\begin{prop}
\label{div1}
For every $A \in M_{k,n}(F)$ and $B \in M_{l,n}(F)$ we have that $A \succeq_n B$ iff $A=CB$ for some bounded $C \in M_{k,l}(F)$.
\end{prop}

\begin{proof}
If $A=CB$ for some bounded $C \in M_{k,l}(F)$ then  $A=CB \succeq_n I_l B =B$.

Conversely, suppose that $A \succeq_n B$. Since $F^l=\im B \oplus (\im B)^\perp$ w.r.t. to the standard inner product,
we can decompose every $v \in F^l$ as $v=Bx+y$, $y^T Bx=0$. Let us define $Cv=Ax$. To show that $C$ is well-defined note that
$Bx+y=Bx'+y'$ implies $Bx=Bx'$ and $y=y'$. Thus, $x-x' \in \ker B \subseteq \ker A$ which implies  $Ax=Ax'$.
To show that $C$ is bounded, pick any $v \in F^l$ and decompose it as $v=Bx+y$ with $y^T Bx=0$. We have that 
$v^T C^T C v= x^T A^T A x\le r x^T B^T B x \le r x^T B^T B x+r y^T y=r (x^T B^T B x+y^T B x+x^T B^T y+y^T y)=r v^T v$.
It follows that $C^T C \le_l r I_l$.
\end{proof}

\begin{prop}
\label{div2}
For every $A \in M_{k,n}(F)$ and $B \in M_{l,n}(F)$ where $k \ge l$, we have that $A \sim_n B$ iff $A=CB$ for some bibounded $C \in M_{k,l}(F)$.
\end{prop}

\begin{proof}
Suppose that $A=CB$ for some bibounded $C \in M_{k,l}(F)$. 
Since $C \sim_l I_l$, it follows that $A=CB \sim_n I_l B =B$.

Since $s I_l \le_l C^T C \le_l r I_l$ for some $r,s \in \NN$,
it follows that $s B^T B_n \le B^T C^T C B \le_n r B^T B$ which implies that $A \sim_n B$. 

Conversely, if $A \sim_n B$, then $\ker A=\ker B$. It follows that 
$t:= \dim (\im B)^\perp = l- \dim  \im B=l -n+\dim \ker B \le k-n+\dim \ker A=\dim (\im A)^\perp =: t'$.
Pick an orthogonal basis $u_1,\ldots,u_t$ of $(\im B)^\perp$ and an orthogonal basis $v_1,\ldots,v_{t'}$ of $(\im A)^\perp$. 
Now make all $u_i$ and $v_j$ bibounded by dividing them with suitable scalars. 
Every element $z \in F^l$ can be written as $z=Bx +\sum_{i=1}^t \alpha_i u_i$. 
We define $Cz:=Ax+\sum_{i=1}^t \alpha_i v_i$. Since $\ker A=\ker B$, $C$ is well-defined and one-to-one. 

Let us show that $C$ is bibounded.
Since $A \sim_n B$ and $u_i \sim_1 1 \sim_1 v_i$ for every $i=1,\ldots,t$, we can pick $r,s \in \NN$
such that $s B^T B \le_n A^T A \le_n r B^T B$ and $s u_i^T u_i \le v_i^T v_i \le r u_i^T u_i$ for every $i=1,\ldots,t$. 
We claim that $s z^T z \le z^T C^T C z \le r z^T z$ for every $z \in F^l$ which implies that $s I_l \le C^T C \le rI_l$. 
Pick any $z=Bx+y \in F^l$ where $y=\sum \alpha_i u_i \in (\im B)^\perp$. We have that 
$z^T z =x^T B^T B x+y^T B x+x^T B^T y+y^T y=x^T B^T B x +y^T y = x^T B^T B x+\sum_{i=1}^t \alpha_i^2  u_i^T u_i$. 
On the other hand, $z^T C^T C z=x^T A^T A x+(C y)^T A x+(A x)^T C y+y^T C^T C y
=x^T A^T A x +y^T C^T C y =x^T A^T A x+\sum_{i=1}^t \alpha_i^2 v_i^T v_i$.
The claim follows.
\end{proof}

\subsection{A characterization of bounded and bibounded matrices}

We will characterize bounded and bibounded matrices in terms of their entries; see Proposition \ref{elementwise}.
We start with some preparation. The following is well-known.

\begin{lemma}
\label{invert}
If $A,B$ belong to $S_n^+(F)$ and $A \le_n B$ and $A$ is invertible, then
$B$ is also invertible, $A^{-1},B^{-1} \in S_n^+(F)$ and $B^{-1} \le_n A^{-1}$.
\end{lemma}

\begin{proof}
Note that $B$ is invertible since $A \le_n B$ implies $\ker A \supseteq \ker B$ by Lemma \ref{kerpsd}.
Now use $A^{-1}-B^{-1}=(A^{-1}-B^{-1})A(A^{-1}-B^{-1})+B^{-1}(B-A)B^{-1}$.
\end{proof}

Let us characterize the analogues of bounded and bibounded elements in $S_n^+(F)$.

\begin{lemma}
\label{auxbounded}
For every matrix $A \in S_n^+(F)$ we have the following.
\begin{enumerate}
\item $A \sqsupseteq_n I_n$ iff all diagonal entries of $A$ are bounded.
\item $A \approx_n I_n$ iff $A \sqsupseteq_n I_n$ and $\det A$ is bibounded.
\end{enumerate}
\end{lemma}

\begin{proof}
One direction of claim (1) is clear and the other follows from $A \le (\tr A)I_n$.
Suppose that $A \sqsupseteq_n I_n$ and $\det A$ is bibounded. By (1) all diagonal entries of $A$ are bounded.
Since $a_{ij}^2 \le a_{ii} a_{jj}$ for all $i,j$, it follows that  nondiagonal entries of $A$ are also bounded.
Therefore all minors of $A$ are bounded. The assumption that $\det A$ is bibounded implies that $A^{-1}=(\det A)^{-1} \mathrm{Cof}(A)$
exists and all its entries are bounded. By (1) it follows that $A^{-1} \sqsupseteq_n I_n$ which implies that $I_n \sqsupseteq_n A$. 

To prove the other direction of claim (2), suppose that $I_n \sqsupseteq_n A$. 
For $k=1,\ldots,n$ write $A_k$ for the upper left $k \times k$ corner of $A$.
By assumption, there exists $s \in \NN$ such that $s I_n \le_n A$, and so $s I_k \le_n A_k$ for all $k$.
Lemma \ref{invert} implies that $A_k$ is invertible and $A_k^{-1} \le s^{-1} I_k$ for every $k$.
Thus, $\frac{\det A_{k-1}}{\det A_k} =(A_k^{-1})_{kk} \le s^{-1}$ for all $k$.
Now, $(\det A)^{-1}$ is bounded since 
$(\det A)^{-1}=(\det A_1)^{-1} \prod_{k=2}^n \frac{\det A_{k-1}}{\det A_k} \le s^{-n}$.
\end{proof}

Proposition \ref{elementwise} is a corollary of Lemma \ref{auxbounded}.

\begin{prop}
\label{elementwise}
For every element $A \in M_{m,n}(F)$ we have the following.
\begin{enumerate}
\item $A$ is bounded iff all entries of $A$ are bounded.
\item $A$ is bibounded iff $A$ is bounded, $m \ge n$, and at least one of the $n \times n$ minors of $A$ is bibounded.
\end{enumerate}
A square matrix $A$ is bibounded iff it is invertible and both $A$ and $A^{-1}$ are bounded.
\end{prop}

\begin{proof}
The diagonal entries of $A^T A$ are $\sum_{j=1}^n a_{ij}^2$ where $i=1,\ldots,n$. They are bounded iff all $a_{ij}$ are bounded.
Claim (1) now follows from Lemma \ref{auxbounded}. To prove claim (2) note that, by Lemma \ref{auxbounded}, a matrix 
$A \in M_{m,n}(F)$ is bibounded iff it is bounded and $\det A^T A$ is bibounded. If $m <n$, then $\det A^T A=0$, so it is not bibounded.
The Binet-Cauchy theorem (see \cite{Ga}, for example) implies that $\det A^T A$ is equal to the sum of squares of all $n \times n$ minors of $A$.
It follows that $\det A^T A$ is bibounded iff the $n \times n$ minor of $A$ of the largest absolute value is bibounded
iff at least one of the $n \times n$ minors of $A$ is bibounded. 
\end{proof}

Proposition \ref{elementwise} implies that a matrix $A \in M_n(F)$ is bounded iff $w(A) \ge 0$. 
However, condition $w(A)=0$ is not equivalent to biboundedness.

\subsection{Bibounded elementary matrices}

Recall that the Gauss algorithm consists of a series of elementary row tranformations that can be represented
as left multiplications by elementary matrices $E_{ij}(\alpha)$, $E_i(\alpha)$ and $P_{ij}$ where  $\alpha \in F$.

\begin{lemma}
\label{elmat}
For every $\alpha \in F$, we have the following.
\begin{enumerate}
\item A matrix of the form $E_{ij}(\alpha)$ is bibounded iff $\alpha$ is bounded.
\item A matrix of the form $E_i(\alpha)$ is bibounded if $\alpha$ is bibounded.
\item A matrix of the form $P_{ij}$ is always bibounded.
\end{enumerate}
\end{lemma}

\begin{proof}
(1) By Lemma \ref{elementwise}, $E_{ij}(\alpha)$ is bounded iff $\alpha$ is bounded. Since $E_{ij}(\alpha)^{-1}=E_{ij}(-\alpha)$ and 
$\alpha$ is bounded iff $-\alpha$ is bounded, the claim follows.

(2) By Lemma \ref{elementwise}, $E_i(\alpha)$ is bounded iff $\alpha$ is bounded. Since $E_i(\alpha)^{-1}=E_i(1/\alpha)$,
the claim follows from the definition of a bibounded element of $F$.

(3) Follows from Lemma \ref{elementwise}.
\end{proof}

We already know that for every $A \in M_{k,n}(F)$ there exists a bibounded $Q \in M_k(F)$ such that $QA$ is a row echelon form.
We can also prove this result by a bibounded version of Gauss algorithm.

If the first column of $A$ contains only zeros then move to the next column.
Otherwise pick in the first column an element of the largest absolute value (the first pivot) and move it to the first row by an appropriate $P_{1j}$.
Now kill all elements below the pivot by  $E_{j1}(-a_{j1}/a_{11})$ where $j=2,\ldots,n$. By the choice
of the pivot, the elements $-a_{j1}/a_{11}$ are bounded. Thus the matrices $E_{j1}(-a_{j1}/a_{11})$ are bibounded.
Finally move to the next column. If all elements from $a_{22}$ to $a_{n2}$ are zero, then move to the next column.
Otherwise pick an element of the largest absolute value (the second pivot) and move it to the second row by an appropriate $P_{2j}$.
Now kill all elements below the pivot with $E_{j2}(-a_{j2}/a_{22})$, $j=3,\ldots,n$, which are bibounded by the choice of the pivot.
Finally, move to the next column. Continue until you run out of columns.

The elements above a pivot cannot be killed unless their valuation is greater or equal to the valuation of the pivot.
This is already clear on the matrix $E_{12}(a)$. However, if $A$ is bibounded and square, then the corresponding row echelon form
is also bibounded and square. By Lemma \ref{elementwise} this can only happen if it has bounded entries and
bibounded determinant. It follows that its diagonal entries are also bibounded. Therefore we can use a bibounded version of
Gauss algorithm to kill all elements above the diagonal. Finally, a bibounded diagonal matrix is clearly a product
of bibounded $E_i(\alpha)$. This proves Proposition \ref{elbiprod}.

\begin{prop}
\label{elbiprod}
A square matrix over $F$ is bibounded iff it is a product of bibounded elementary matrices.
\end{prop}

\begin{remark}
For every $A \in M_{\ast,n}(F)$ and every row echelon form $B \in M_{\ast,n}(F)$,
\begin{center}
$A \succeq_n B$ iff $\left[ \begin{array}{c} B \\ A \end{array} \right] \sim_n \left[ \begin{array}{c} B \\ 0 \end{array} \right]$
iff bibounded Gauss on $\left[ \begin{array}{c} B \\ A \end{array} \right]$ gives $\left[ \begin{array}{c} B \\ 0 \end{array} \right]$.
\end{center}
\end{remark}

\section{Archimedean canonical forms}
\label{seccan}

We want to choose a canonical representative in each archimedean class of $M_{\ast,n}(F)$. 
This works for Laurent series fields but it does not work in general.

Let $\Gamma$ be a linearly ordered Abelian group and let $F = \RR((\Gamma))$ be the field of formal Laurent series
with real coefficients and with exponents in $\Gamma$. 
A row echelon form $C \in M_{\ast,n}(F)$ with pivots $c_{i,k_i}$, $i=1,\ldots,r$, is called an \textit{archimedean canonical form} if
it has no zero rows and there exist $m_1,\ldots,m_r \in \Gamma$ such that 
\begin{enumerate}
\item[(i)] $c_{i,k_i}=t^{m_i}$ and
\item[(ii)] for every $j<k_i$, $c_{j,k_i}$ has no terms with exponents $\ge m_i$.
\end{enumerate}
Explicitly,
$$C=\left[\begin{array}{ccccccccccc}
0 & \ldots & t^{m_1} & c_{1,k_1+1} & \ldots & c_{1,k_2} & c_{1,k_2+1} & \ldots & c_{1,k_3} & c_{1,k_3+1} & \ldots \\
0 & \ldots & 0 & 0 & \ldots & t^{m_2} & c_{2,k_2+1} & \ldots & c_{2,k_3} & c_{2,k_3+1} & \ldots \\
0 & \ldots & 0 & 0 & \ldots & 0 & 0 & \ldots & t^{m_3} & c_{3,k_3+1} & \ldots \\
\vdots &   & \vdots & \vdots &  & \vdots & \vdots &  & \vdots & \vdots & \ddots \\
\end{array}\right]$$

Note that an archimedean canonical form contains more information than just the information about
the archimedean classes of its entries.

\begin{prop}
For every nonzero $A \in M_{\ast,n}(F)$ there exists a unique archimedean canonical form $C \in M_{\ast,n}(F)$ such that $A \sim_n C$. 
\end{prop}

\begin{proof}
By Proposition \ref{qr} we may assume that $A$ is a row echelon form with no zero rows.
Each pivot $a_{i,k_i}$ of $A$ can be decomposed uniquely as $a_{i,k_i}=u_i t^{m_i}$ where $u_i$ is bibounded.
If we divide each nonzero row of $A$ with $u_i$ we get a matrix $B$ such that $B \sim_n A$ and $b_{i,k_i}=t^{m_i}$. Now we perform bibounded row
transformations of $B$ which use the pivot $t^{m_i}$ to kill all terms with exponents $\ge m_i$ in all entries above the pivot. 
The result is an archimedean canonical form $C$ such that $C \sim_n B$. This proves the existence part.

To prove uniqueness pick another archimedean canonical form $D \sim_n A$.
By Lemma \ref{ker}, $C$ and $D$ have the same rank and thus the same number of rows, say $r$.
By Corollary \ref{corshape}, $C \sim_n D$ implies that $C$ and $D$ have the same
shape and the same pivots in each row. By Proposition \ref{div2}, there exists a bibounded $Q \in M_r(F)$ such that
$D= Q C$. If $q_1,\ldots,q_r$ are the columns of $Q$ and $e_1,\ldots,e_r$ are the columns of $I_r$, 
then $D= Q C$ implies that
\begin{eqnarray}
\label{bla1}
t^{m_1} q_1 & =  & t^{m_1} e_1 \\
\label{bla2}
c_{1,k_2} q_1+t^{m_2} q_2 & = & d_{1,k_2} e_1+t^{m_2} e_2 \\
\label{bla3}
c_{1,k_3} q_1+c_{2,k_3}q_2+t^{m_3} q_3 & = & d_{1,k_3} e_1+d_{2,k_3} e_2+t^{m_3} e_3 \\
\notag & \vdots &
\end{eqnarray}
Equation \eqref{bla1} implies that $q_1=e_1$. Equation \eqref{bla2} implies that $(c_{1,k_2}-d_{1,k_2})e_1=t^{m_2}(e_2-q_2)$.
Since all powers that appear in the components of the left-hand side are $<m_2$ and all powers that appear
in the components of the right-hand side are $\ge m_2$, it follows that $c_{1,k_2}=d_{1,k_2}$ and $q_2=e_2$.
Similarly, equation \eqref{bla3}  implies that $(c_{1,k_3}-d_{1,k_3})e_1+(c_{2,k_3}-d_{2,k_3})e_2=t^{m_3}(e_3-q_3)$.
Since all powers that appear in the components of the left-hand side are $<m_3$ and all powers that appear
in the components of the right-hand side are $\ge m_3$, it follows that $c_{1,k_3}=d_{1,k_3}$, $c_{2,k_3}=d_{2,k_3}$ and $q_3=e_3$.
This process stops after $r$ steps and gives that $Q=I_r$. 
\end{proof}

\begin{remark}
Since every ordered field $F$ with value group $\Gamma$ can be $o$-embedded into $\RR((\Gamma))$ by Hahn's embedding theorem,
it is tempting to assume that archimedean canonical forms exist for matrices over general fields.
However, the problem is that to construct the archimedean canonical form we do not need just ring operations 
but also the operation of truncation which can take us out of the image of $F$ in $\RR((\Gamma))$.
If we want the image of $F$ to be truncation closed we must pass to real closures.
Let $\bar{F}$ be the real closure of $(F,\le)$ and let $\bar{\Gamma}$ be he division hull of $\Gamma$.
In \cite{kuhlmann} it is shown that there exists an $o$-embedding $\psi \colon \bar{F} \to \RR((\bar{\Gamma}))$
such that $\psi(\bar{F})$ is truncation closed in $\RR((\bar{\Gamma}))$ and $\psi(F)$ is contained in $\RR((\Gamma))$.
Pick now any $A \in M_n(F)$ and compute the archimedean canonical form $C$ of $\psi_n(A):=[\psi(a_{ij})]$ in $M_n(\RR((\Gamma)))$.
The properties of $\psi$ imply that there exists $B \in M_n(\bar{F})$ such that $\psi_n(B)=C$.
Since $\psi_n$ is an o-embedding, it follows that $B \sim_n A$ in $M_n(\bar{F})$.
So $C$ is a representative of the archimedean class of $A$ in $M_n(\bar{F})$.

This is not very useful because in $M_n(\bar{F})$ we have better representatives of archimedean classes.
Namely, if $\sum_i \sigma_i P_i$ is the spectral decomposition of $A^T A$ and $\tau_i$
are representatives of archimedean classes of $\sqrt{\sigma_i}$, then $A \sim \sum_i \tau_i P_i$.
\end{remark}

\section{Can we multiply archimedean classes?}
\label{seclast}

We already know that the relation $\sim_n$ is not a congruence relation on the multiplicative semigroup on $M_n(F)$.
More precisely, $A \sim_n B$ always implies $AC \sim_n BC$ but it does not always imply $CA \sim_n CB$.
Therefore, the factor set $M_n(F)/\!\sim_n=\Gamma_n \cup \{\infty\}$ is not a semigroup in the usual way.
We can address this issue in at least three different ways:
\begin{enumerate}
\item We give up on the multiplicative structure of $M_n(F)/\!\sim_n$ and consider the following construction instead.
For every $C \in M_n(F)$, the mapping $$\phi_C \colon M_n(F)/\!\sim_n \to M_n(F)/\!\sim_n, \quad \phi_C(v_n(A)):=v_n(AC)$$ is order-preserving.
The mapping $\phi \colon C \mapsto \phi_C$ from the multiplicative semigroup of $M_n(F)$ to the semigroup
of order-preserving maps from $M_n(F)/\!\sim_n$ to $M_n(F)/\!\sim_n$ is a homomorphism of semigroups.
\item We define some unnatural multiplication on $M_n(F)/\!\sim_n$; see section \ref{secmult}.
\item We modify the definition of the relation $\sim_n$; see section \ref{secgg}.
\end{enumerate}

\subsection{An unnatural multiplication}
\label{secmult}

We will define a multiplication on $M_n(F)/\sim_n$. Although the definition
is unappealing its properties are very good.

\begin{prop}
The operation $\square$ on $M_n(F)$ defined by
$$A \square B := \m{A} \m{B} I_n$$
has the following properties:
\begin{enumerate}
\item It is associative and commutative.
\item If $A \succeq_n B$ for some $A,B \in  M_n(F)$ then $A \square C \succeq_n B \square C$ for every $C \in  M_{\ast,n}(F)$.
In particular $\sim_n$ is a congruence relation on  $(M_n(F),\square)$.
\item $v_n$ is surmultiplicative, i.e. $v_n(A  B) \succeq_n  v_n(A) \square v_n(B) := v_n(A \square B)$ 
for every $A\in M_n(F)$ and every $B \in  M_n(F)$.
\item $\phi$ is multiplicative, i.e. $\phi(v_n(A) \square v_n(B))=\phi(v_n(A))\phi(v_n(B))$ for every $A, B \in  M_n(F)$.
\item $\wedge_n$ is distributive over $\square$, i.e. $(v_n(A) \wedge_n v_n(B)) \square v_n(C) = (v_n(A) \square v_n(C)) \wedge_n (v_n(B) \square v_n(C))$
for every $A,B,C \in M_n(F)$.
\item $\vee_n$ is not distributive over $\square$.
\end{enumerate}
\end{prop}

\begin{proof}
Claim (1) is clear since $(A \square B) \square C = \m{A} \m{B} \m{C} I_n = A \square (B \square C).$
Claim (2) follows from the fact that $A \succeq_n B$ implies $\m{A} \succeq \m{B}$. 

For every $A \in M_{\ast,n}(F)$ we have that $A^T A \le_n (\tr A^T A)I_n \le n^2 \m{A}^2 I_n$, so
\begin{equation}
\label{maxeq}
A \succeq_n \m{A} I_n.
\end{equation}
Claim (3) follows from inequality \eqref{maxeq}. Namely,
$AB \succeq_n (\m{A}I_n)B=B(\m{A}I_n) \succeq_n (\m{B}I_n)(\m{A}I_n)$.

Claim (4) follows from $\phi(v_n(A) \square v_n(B))=\phi(v_n(A \square B))=v(\m{A \square B})=v(\m{A} \m{B})
= v(\m{A})v(\m{B})=\phi(v_n(A))\phi(v_n(B))$.

If we multiply $$\max\{\m{A},\m{B}\}^2 \le \m{A}^2+\m{B}^2 \le 2\max\{\m{A},\m{B}\}^2$$ with $\m{C}^2$ and we use 
$\m{ \left[\begin{array}{c} A \\ B \end{array} \right] }=\max\{\m{A},\m{B}\}$ we get that
$$\left[\begin{array}{c} A \\ B \end{array} \right] \square C \sim_n \left[\begin{array}{c} A \square C \\ B \square C \end{array} \right].$$
This implies claim (5). 

To prove claim (6), take 
$$A=\left[\begin{array}{cc} 1 & 0 \\ 0 & 0 \end{array} \right], \quad B=\left[\begin{array}{cc} 0 & 0 \\ 0 & 1 \end{array} \right]$$
and any $C$ with $\m{C}=1$. Then
$$(v_2(A) \vee_2 v_2(B)) \square v_2(C)  = 0 \square C=0$$
and
$$(v_2(A) \square v_2(C)) \vee_2 (v_2(B) \square v_2(C))=I_2 \vee_2 I_2=I_2. \qedhere$$
\end{proof}

It would be interesting to know if there exists a multiplication on $\Gamma_n$ such that we also have distributivity for $\vee_n$.

\subsection{A variant of relations $\succeq_n$ and $\sim_n$}
\label{secgg}

Let $F$ be an ordered field and $n$ an integer. For every $A,B \in M_n(F)$ we write
\begin{center}
$A \gg_n B$ iff $C A \succeq_n C B$ for every $C \in M_n(F)$. \\
$A \equiv_n B$ iff $A \gg_n B$ and $B \gg_n A$.
\end{center}
It is clear that $\equiv_n$ is a congruence relation on the multiplicative semigroup of $M_n(F)$ 
and that the congruence class of zero is a singleton.

\begin{prop}
\label{propgg}
For every $A,B \in M_n(F)$, the following are equivalent.
\begin{enumerate}
\item $A \gg_n B$.
\item For every $y \in F^n$ there exists $r \in \NN$ such that $A^T yy^T A \le_n r B^T yy^T B$.
\item For every $y \in F^n$ there exists a bounded $\alpha_y \in F$ such that $y^T A= \alpha_y y^T B$.
\item There exists a bounded $\alpha \in F$ such that $A=\alpha B$.
\end{enumerate}
\end{prop}

\begin{proof}
(1) implies (2). Just replace $C$ with a matrix whose first row is $y^T$ and other rows are zero.

(2) implies (3). Use that for every vectors $w,z \in F^n$ which satisfy $ww^T \le_n r zz^T$
for some $r \in F$ there exists $\alpha \in F$ such that $w=\alpha z$ and $\alpha^2 \le r$.

(3) implies (4). By assumption $A^T$ and $B^T$ are locally linearly dependent. 
By Theorem 2.3 in \cite{bs} either $A^T$ and $B^T$ are linearly dependent or their exists a nonzero $v \in F^n$
such that $\im A^T $ and $\im B^T$ are subsets of the span of $v$. If either $\rank A>1$ or $\rank B >1$
then we are in the first case. Consequently $A= \alpha B$ for some $\alpha$.
If $\rank A \le 1$ and $\rank B \le 1$ then there exist $a,b \in F^n$ such that 
$A^T=v a^T$ and $B^T= v b^T$. By (3), for every $y \in F^n$ there exists a bounded $\alpha_y \in F$
such that $(y^T a) v^T =\alpha_y(y^T b) v^T$. Since $v \ne 0$, it follows that $y^T a=\alpha_y y^T b$.
Therefore $y^T b=0$ implies $y^T a=0$ for every $y \in F^n$. In particular, $a= \alpha b$.
Therefore $A=\alpha B$ in this case, too. By (3), $\alpha$ must be bounded.

(4) implies (1). This is clear.
\end{proof}

The following is now clear:

\begin{cor}
For every matrices $A,B \in M_n(F)$ we have that
$A \equiv_n B$ iff either $A=B=0_n$ or if $A \ne 0_n$ and $B \ne 0_n$ and 
$\frac{1}{\Vert A \Vert_\infty} A=\frac{1}{\Vert B \Vert_\infty}B$ and $v_F(\Vert A \Vert_\infty)
= v_F(\Vert B \Vert_\infty)$. In other words, the factor set $M_n(F)^{\ne 0_n}/ \equiv_n$ is isomorphic to
$K_n \times \Gamma_F$ where $K_n=\{A \in M_n(F) \mid \Vert A \Vert_\infty=1\}$.
\end{cor}

\end{document}